\documentclass{article}
\usepackage{pack}
\usepackage{amsmath,amssymb}
\begin{document}
% Numbering equations by section
\def\theequation{\thesection.\arabic{equation}}
\makeatletter
\@addtoreset{equation}{section}
\makeatother
% Numbering figures by section
\def\thefigure{\thesection.\arabic{figure}}

%Proof environment
\newenvironment{proof}[1][Proof]{\textbf{#1.} }{\hfill \rule{0.5em}{0.5em}}

%Remark environment
\newtheorem{remark}{Remark}[section]

\renewcommand{\titre}{Stochastic chemical kinetics with energy parameters}
\renewcommand{\smalltitre}{Stochastic chemical kinetics}
\renewcommand{\auts}{Guy Fayolle, Vadim Malyshev, Serguei Pirogov}
\tit

\begin{abstract}
We introduce new models of energy redistribution in stochastic
chemical kinetics with several molecule types and energy
parameters. The main results concern the situations when there are
product form measures. Using a probabilistic interpretation of the
related Boltzmann equation, we find some invariant measures  
explicitly and we prove convergence to them.
\end{abstract}

\index{Chemical Kinetics}
\index{Boltzmann equation}
\index{Markov chain}
\index{Product form}

%%%%%%%%%%%%%%%%%%
\section{Introduction}
Metabolic pathways in molecular biology are chains or networks of
chemical reactions providing redistribution of energy, in particular
synthesis of ATP molecules, universal energy stocks in cells. Here we
elaborate simple models of energy redistribution. According to a
classical approximation, the energy of a molecule can be subdivided in
two parts: internal (chemical) energy and kinetic energy. The model is
the following.

\noindent
Assume that there are $V$ molecule types $v\in
\left\{1,\ldots,V\right\} $, $n_{v}(t)$ molecules for each type $v$ at
time $t$. Types $v$ can be interpreted as chemical substances with
different formulas, different isomers of the same formula, or even as
different energy levels (spectrum) of the same molecule.

\noindent 
The total number of molecules $M=\sum_{v}n_{v}(t)$ will be conserved.
A molecule may be characterized by a pair $(v,T),v=1,\ldots,V$, where
$T\in R_{+}$ is the kinetic energy of the molecule. Then each molecule
of type $v$ at time $t$ has energy
\begin{equation*}
E(t)=I(v)+T(t),
\end{equation*}
where $I(v)$ is the internal (or chemical) energy of any molecule of
type $v$, $T(t)$ being the kinetic energy of a concrete molecule at time
$t$. Thus, for any $v,t$, $I(v)$ are fixed numbers and $T(t)$ are random.

\noindent
We use the approach usually refered to as \emph{stochastic chemical
kinetics}. It appeared in physical papers, see \cite{Leo}, but was
also explored also by mathematicians for many models with small $V$,
(see e.g. the reviews \cite{McQua, Kal}). However these models did not
consider any energy parameter. Independently of this, Kac \cite{Kac}
considered a beautiful model with mean field collisions. Deeper
results in this model appear even recently, see \cite{CaCaLo}.
However, in Kac's model molecules were characterized only by kinetic
energies, that is $V=1$. Our model can be considered as a mixture of
these two: there are molecule types and energy parameter.

\medskip\noindent 
The plan of the paper is as follows. In section 2, we introduce our
probabilistic microscopic model and provide the corresponding Boltzmann type
equation. Proof of the finite microtime scaling limit convergence to this
equation uses standard technical tools and will be published elsewhere. In
section 3 we get deeper results for the one type case with uniform
scattering: find invariant measures and prove convergence of the Boltzmann
equation for large macrotime. In section 4 we provide many examples, with
similar results for multitype models.

%%%%%%%%%%%%%%%%%%%%%%
\section{Finite time scaling limit}
Unless otherwise stated, we consider a system of binary reactions of
the form $A+B\rightarrow C+D$. We assume energy conservation and
random momentary collisions, that is when a pair of different
molecules $(v,T),(v^{\prime },T^{\prime })$ collide at time $t$ then a
new pair $(v_{1},U),(v_{1}^{ \prime },U^{\prime })$ appears at time
$t+0$, so that
\begin{equation*}
I(v)+T+I(v^{\prime })+T^{\prime }=I(v_{1})+U+I(v_{1}^{\prime
})+U^{\prime}.
\end{equation*}
Obviously, the reaction is possible only if 
\begin{equation}
I(v)+T+I(v^{\prime })+T^{\prime }\geq I(v_{1})+I(v_{1}^{\prime }).
\label{compa_1}
\end{equation}
We define the following continuous time Markov chain. The state is an
array of $V$ vectors $((v,T_{i}),i=1,\ldots,n_{v}),v\in V$. Thus,
their total length $M=\sum_{v=1}^{V}n_{v}$ is conserved, but not
necessarily $n_{v}$. The order of components in each vector
$((v,T_{i}),i=1,\ldots,n_{v})$ does not play any role, so that we will
consider only functions symmetric in the vector coordinates.

\noindent
On the time interval $(t,t+dt)$, each pair of molecules
$(v,T),(v^{\prime },T^{\prime})$ has a collision with probability
$\frac{1}{M}\alpha _{vv^{\prime }}(T,T^{\prime })dt$. The functions
$\alpha _{vv^{\prime }}(x,y)$ are assumed to be bounded and smooth on
$R_{+}^{2}$. As a result of this collision, some pair
$(v_{1},U),(v_{1}^{\prime },U^{\prime })$ appears, provided that
condition (\ref{compa_1}) holds for at least one pair
$(v_{1},v_{1}^{\prime })$. Otherwise nothing occurs. The distribution
of the new pair is defined by the rules listed hereafter. For any
$v_{1},v_{1}^{\prime},v,v^{\prime },T,T^{\prime }$, the conditional
densities
\begin{equation*}
P\bigl((v_{1},U),v_{1}^{\prime }|(v,T),(v^{\prime },T^{\prime
})\bigr)\geq 0
\end{equation*} are supposed to satisfy the following properties.
\begin{enumerate}
\item If 
\begin{equation*}
I(v)+T+I(v^{\prime })+T^{\prime }<I(v_{1})+I(v_{1}^{\prime}),
\end{equation*}
then 
\begin{equation*} 
P\bigl((v_{1},U),v_{1}^{\prime
}|(v,T),(v^{\prime },T^{\prime })\bigr)=0.
\end{equation*}
\item For any $v,v^{\prime},v_{1},v_{1}^{\prime},T,T^{\prime}$, the
density function 
\[f(U)=P\bigl((v_{1},U),v_{1}^{\prime}|(v,T),(v^{\prime},T^{\prime})\bigr)
\]
is defined on the interval
$I=\left[0,I_{v}+T+I_{v^{\prime}}+T^{\prime}-I(v_{1})-I(v_{1}^{\prime})\right]
$ and
\[
\sum_{v_{1},v_{1}^{\prime}}
\int_{I}P((v_{1},U),v_{1}^{\prime}|(v,T),(v^{\prime },T^{\prime}))dU=1.
\]
\end{enumerate}
Thus the distribution of the triple $(v_{1},U,v_{1}^{\prime})$ is
entirely defined by
\[P\bigl((v_{1},U),v_{1}^{\prime}|(v,T),(v^{\prime},T^{\prime })\bigr) 
\]
and
$U^{\prime}=I(v)+T+I(v^{\prime})+T^{\prime}-(I(v_{1})+U+I(v_{1}^{\prime})).$ 

\medskip\noindent
Hence, for $V$ finite sets
$\left\{T_{v,1},\ldots,T_{v,n_{v}}\right\},v=1,\ldots,V$, we have
defined a Markov process on $R_{+}^M$, which will be denoted by
$\mathcal{L}_{M}$. It is worth remarking that, when the total energy $U$
is fixed, $\mathcal{L}_{M}$ has a compact state space. Then, under
some nondegeneracy conditions on $\alpha $ and $P$, this Markov chain
for fixed $M$ approaches, as $t\rightarrow \infty $, its unique
stationary distribution $\pi ^{(M)}(U)$. Our goal will be to study, under
some conditions, the scaling limit $M\rightarrow \infty $ for fixed
$t$, and also the large time limit $t\rightarrow \infty $.

\medskip\noindent
Let $n_{v}^{(M)}(A,t)$ denote the number of type $v$ molecules at
time $t$ having kinetic energy $T$ in the set $A\subset R_{+}$. In the
limit $ M\rightarrow \infty $ we have to impose initial conditions at
time zero
\begin{equation*}
\lim_{M\rightarrow \infty }\frac{n_{v}^{(M)}(A,0)}{M}=\int_{A}\rho
_{v}(x,0)dx ,
\end{equation*} 
for some nonnegative functions $\rho _{v}(x,0),
\sum_{v}\int_{R_{+}}\rho _{v}(x,0)dx=1$, called \emph{concentrations}.
Our goal is to prove that, as $M\rightarrow \infty $, the sequence of
Markov processes $\mathcal{L}_{M}$ converges to some deterministic
evolution $\mathcal{L}$ of the concentrations. We state now our first result.
\begin{theo} \label{TH1}
For any $A$ and $t$, there exist deterministic limits (in probability) 
\begin{equation*}
\lim_{M\rightarrow \infty }\frac{n_{v}^{(M)}(A,t)}{M}=\int_{A}\rho
_{v}(x,t)dx ,
\end{equation*} 
where the $\rho _{v}(x,t)$'s are some non-negative functions
satisfying the following Boltzmann type equations
\begin{equation}\label{bolt_1}
\begin{split}
\frac{\partial\rho_{v_{1}}(x,t)}{\partial t} &= \sum_{v,v^{\prime},v_{1}^{\prime}}\int_{R_{+}^{2}}\bigl[\alpha_{vv^{\prime}}(y,z)
P\bigl((v_{1},x),v_{1}^{\prime}|(v,y),(v^{\prime},z)\bigr)\rho _{v}(y,t)\rho _{v^{\prime}}(z,t) \\ 
& \qquad -\alpha_{v_{1}v_{1}^{\prime}}(x,z)P((v,y),v^{\prime}|(v_{1},x),(v_{1}^{\prime},z))
\rho_{v_{1}}(x,t)\rho _{v_{1}^{\prime}}(z,t)\bigr] dydz ,\\
\end{split}
\end{equation}
with the initial condition $\rho _{v}(x,0)$.
\end{theo}

\paragraph{Other reaction types}
Quite similarly one can consider other types of reactions. For example
consider the reaction $A\rightarrow B+C$. In this case on the time interval $ 
(t,t+dt)$ each molecule $(v,T)$ with probability $\alpha _{v}(T)dt$ is
transformed into two molecules (note that the scaling is different here). The
distribution of the products $(v_{1},U),(v_{1}^{\prime },U^{\prime })$ is
defined by similar kernels $P((v_{1},U),v_{1}^{\prime }|(v,T))$ under the
condition 
\begin{equation*}
I_{v_{1}}+I_{v_{1}^{\prime}}\leq I_{v}+T.
\end{equation*}

%%%%%%%%%%%%%%%%%%%%%%%%%
\section{One type case}
%%%%%%%%%%%%%%%%%
\subsection{Probabilistic interpretation}
We consider in this section the particular situation with only one
molecule type $v$. It will be also assumed that the rates $\alpha
(T,T^{\prime })=\alpha _{vv}(T,T^{\prime })=\alpha $ and the
conditional probabilities $P(U|T,T^{\prime })$ are uniform on the
interval $\left[ 0,T+T^{\prime }\right] $. It turns out that the
limiting stationary distribution can be found explicitly. Indeed,
equation (\ref{bolt_1}) can be rewritten as
\begin{equation}
\frac{\partial \rho (x,t)}{\partial t}=\alpha \int_{x}^{\infty
}\frac{ds}{s} \int_{0}^{s}\rho(u,t)\rho (s-u,t)du-\alpha \rho (x,t).
\label{bolt_11}
\end{equation} 
[Similar equations appeared in \cite{Ern} in a different context]. Now
one can guess a fixed point: it is $\rho (x)=\beta e^{-\beta x}$, but
it also can be obtained from a very clear probabilistic picture.

\medskip\noindent
Let us consider \emph{finite particle dynamics}, that is the chain
$\mathcal{L}_{M}$, the states of which are finite subsets of $R_{+}$
with $M$ elements.

\noindent
Take first the case $M=2$. Define the chains $\mathcal{L}_{2}(U)$ as
the restriction of $\mathcal{L}_{2}$ on states with total energy $U$.
Then the chains $\mathcal{L}_{2}(U)$ are irreducible and nilpotent:
that is, already after the first jump we get the stationary
distribution $\pi_{2}(U)$, with $T$ uniformly distributed on $[0,U]$
and $T^{\prime }=U-T$. Hence, for any initial condition,
$\mathcal{L}_{2}$ is a mixture of $\mathcal{L}_{2}(U)$. We see that,
for any density $f(U)$, the measure
\begin{equation*}
\int_{R_{+}}\pi _{2}(U)f(U)dU
\end{equation*} 
is an invariant measure for $\mathcal{L}_{2}$. Indeed one of these
invariant measures is of greatest interest to us. Let the random
vector $(\xi_{1},\xi _{2})$ on $R_{+}^{2}$ be defined by the measure
$\mu_{2,\beta}$, such that the two random variables $\xi _{1},\xi
_{2}$ on $R_{+}$ be i.i.d with density $\rho (x)=\beta \exp (-\beta
x)$. Consider a new random vector $(\eta _{1},\eta _{2})$, where
$\eta _{1}$ is picked at random on the interval $[0,\xi _{1}+\xi
_{2}]$ and $\eta _{2}=\xi _{1}+\xi _{2}-\eta _{1}$. This defines a
transformation of measures $\mu _{2,\beta }^{\prime }=W\mu
_{2,\beta}$. In fact we have the following

\begin{lem}
The measure $\mu _{2,\beta }$ is invariant with respect to $W$, that is 
\begin{equation}
\mu _{2,\beta }^{\prime }=\mu _{2,\beta } \label{inv_1}.
\end{equation}
\end{lem}
\begin{proof} Immediate, since the density of $\xi =\xi _{1}+\xi _{2}$ 
is $\beta ^{2}x\exp (-\beta x) $. Then picking a random point on the
interval $\left[0,x\right] $ yields $\mu _{2,\beta}$, whence equality
(\ref{inv_1}) follows.
\end{proof}
In addition, (\ref{inv_1})  gives
\begin{equation*}
\rho (u)=\int_u^{\infty}\frac{dx}{x}\int_{0}^{x}\rho (y)\rho (x-y)dy,
\end{equation*} 
which is exactly the stationary form of equation (\ref{bolt_11}).

\noindent
For $M\geq 3$, the Markov chain $\mathcal{L}_{M}$ has also irreducible
components $\mathcal{L}_{M}(U)$, consisting of all states
$(T_{1},\ldots,T_{M})$ with $T_{1}+\cdots+T_{M}=U$. For fixed $M$ and
$U$ the invariant measure of the chain $\mathcal{L}_{M}(U)$ is the
uniform measure on the simplex $T_{1}+\cdots+T_{M}=U$. An invariant
measure on $\mathcal{L}_{M}$ can be found as follows. 

\noindent
Take $M$ independent particles, having each density $\beta e^{-\beta
x}$ on $R_{+}$ and let $\mu _{M,\beta }$ denote their joint
distribution.
\begin{lem}
The measure $\mu _{M,\beta}$ is invariant for $\mathcal{L}_{M}$.
\end{lem}
\begin{proof} It follows from the previous lemma, because the generator of
$\mathcal{L}_{M}$ is the sum of generators corresponding to all pairs
$ (i,j),i,j=1,\ldots,M,i\neq j$.
\end{proof}
\begin{remark}
  One can show that $\mathcal{L}_{M}(U)$ is reversible, by using the
  classical Kolmogorov's reversibility criteria for Markov with
  transition rates $\lambda _{\alpha \beta }$, namely
\begin{equation*}
\lambda_{\alpha_{1}\alpha_{2}}\lambda_{\alpha_{2}\alpha_{3}}\ldots
\lambda_{\alpha_{k}\alpha_{1}}=\lambda_{\alpha _{1}\alpha_{k}} 
\lambda_{\alpha_{k} \alpha _{k-1}}\ldots\lambda _{\alpha_{2}\alpha_{1}}.
\end{equation*} 
See related questions in \cite{Whi}.
\end{remark}

%%%%%%%%%%%%%%%%%%%%%%%%
\subsection{Convergence for Boltzmann equation}
According to the above section,  when the total initial energy $U$ satisfies the condition
$U=M/\beta$, we have
\begin{equation*}
\lim_{M\rightarrow\infty}\lim_{t\rightarrow
\infty}\frac{n_{v}^{(M)}(A,t) }{M}=\int_{A}\beta e^{-\beta x}dx.
\end{equation*} 

We will consider now the quantity $\lim_{t\rightarrow \infty
}\lim_{M\rightarrow\infty}$.

\begin{theo}
For Boltzmann equation (\ref{bolt_1}), for any initial condition
$\rho(x,0)$, we have
\begin{equation}
\lim_{t\rightarrow \infty }\rho (x,t)=\beta e^{-\beta x},\quad x\geq 0
\label{boltz_12}
\end{equation}
\end{theo}
\begin{proof} The sketch is the following. First, we prove in the
 next subsection, under more general assumptions, that any initial
distribution converges to some fixed point. Secondly, we will show
that there is a unique one-dimensional manifold of fixed points, namely
$\beta e^{-\beta x},0<\beta<\infty$. This will conclude the proof,
since $\beta$ itself is uniquely determined by the initial mean
energy
\begin{equation*}
T(0)=\lim_{M\rightarrow \infty
}\frac{1}{M}\sum_{i=1}^{M}T_{i}(0) = \frac{1}{\beta}.
\end{equation*}
\end{proof}

%%%%%%%%%%%%%%%%%%
\subsection{Local equilibrium condition}
We come back here to an arbitrary number of types. We will say that a
positive function $f(v,x)$ on $V\times R_{+}$ with $\sum_{v}\int
f(v,x)dx=C<\infty $, satisfies a \emph{local equilibrium condition}
(LE) if, for any $\gamma,\gamma _{1}$,
\begin{equation}
\sum_{\gamma ^{\prime },\gamma _{1}^{\prime }}\bigl[w(\gamma ,\gamma _{1}|\gamma
^{\prime },\gamma _{1}^{\prime })f(\gamma ^{\prime })f(\gamma _{1}^{\prime
})-w(\gamma ^{\prime },\gamma _{1}^{\prime }|\gamma ,\gamma _{1})f(\gamma
)f(\gamma _{1})\bigr]=0,  \label{LE}
\end{equation}
where we use the notation
\begin{equation*}
\gamma =(v,x), \quad \sum_{\gamma }=\sum_{v}\int dx,
\end{equation*}
and 
\begin{multline*}
w(\gamma ,\gamma _{1}|\gamma ^{\prime },\gamma _{1}^{\prime })= \\
\alpha _{v^{\prime }v_{1}^{\prime }}(x^{\prime },x_{1}^{\prime
})P\bigl((v,x),v_{1}|(v^{\prime },x^{\prime }),(v_{1}^{\prime
},x_{1}^{\prime })\bigr)\delta (x_{1}-(x^{\prime }+x_{1}^{\prime
}+I_{v^{\prime }}+I_{v_{1}^{\prime }}-x-I_{v}-I_{v_{1}})).
\end{multline*}
One can assume $C=1$. Then, in the one type case, this is tantamount
to saying that $\mathcal{L}_{2}$ has the invariant product form
distribution $f(x)f(y)$.

\medskip\noindent
The \emph{fixed point condition} (FP) 
\begin{equation}
\sum_{\gamma _{1},\gamma ^{\prime },\gamma _{1}^{\prime }}\bigl[w(\gamma ,\gamma
_{1}|\gamma ^{\prime },\gamma _{1}^{\prime })f(\gamma ^{\prime })f(\gamma
_{1}^{\prime })-w(\gamma ^{\prime },\gamma _{1}^{\prime }|\gamma ,\gamma
_{1})f(\gamma )f(\gamma _{1})\bigr]=0, \label{FP}
\end{equation}
valid for any $\gamma$, follows immediately from (\ref{LE}).

\noindent
We shall say that $f(\gamma )$ satisfies a \emph{detailed balance condition}
(DB) whenever
\begin{equation}
w(\gamma ,\gamma _{1}|\gamma ^{\prime },\gamma _{1}^{\prime })f(\gamma
^{\prime })f(\gamma _{1}^{\prime })-w(\gamma ^{\prime },\gamma _{1}^{\prime
}|\gamma ,\gamma _{1})f(\gamma )f(\gamma _{1})=0,  \label{DB}
\end{equation}
for any $\gamma ,\gamma ^{\prime },\gamma _{1},\gamma _{1}^{\prime }$. In
the above one type example, DB condition holds if one chooses
\begin{equation*}
f_{0}=\beta e^{-\beta x},
\end{equation*}
for any positive $\beta$. Note that DB$\rightarrow $LE$\rightarrow $FP.

\noindent
Let us define the \emph{relative entropy} of $f$ with respect to $f_{0}$,
assuming both $f$ and $f_{0}$ are positive. Farther on, $f_{0}$ will
be fixed and therefore omitted in the notation, so that
\begin{equation}
H(f)\equiv H(f,f_{0})=\sum_{\gamma }f(\gamma )\log \left[\frac{f_{0}(\gamma
)}{f(\gamma)}\right].
\label{ent_1}
\end{equation}

\begin{theo}
Assume that there exists some $f_{0}(\gamma )>0$ satisfying the local
equilibrium condition. Then for any initial $f(\gamma ,0)$ with $H(f(.,0))$
finite, the function $f(\gamma )=f(\gamma ,t)$, that is the solution of
equation (\ref{bolt_1}), does satisfy
\begin{equation*}
\frac{dH(f)}{dt}\geq 0.
\end{equation*} 
Moreover, as $t\rightarrow \infty $, $f(\gamma ,t)$ tends to some
fixed point $f_{\infty }$ which depends in general on the initial data
$f(\gamma,0)$. LE condition holds for any stationary solution $f$,
that is for any fixed point of (\ref{bolt_1}).
\end{theo}
\begin{proof} The integrability of $\frac{df(\gamma )}{dt}$ 
follows from (\ref {bolt_1}), so that the following conservation law
holds
\begin{equation}
\sum_{\gamma }\frac{df(\gamma )}{dt}=0  \label{ent_2}.
\end{equation} 
Differentiating (\ref{ent_1}) and using (\ref{ent_2}), we get
\begin{equation*}
\frac{dH(f)}{dt}=\sum_{\gamma }\frac{df(\gamma )}{dt}\log
\left[\frac{f_{0}(\gamma)}{f(\gamma )}\right].
\end{equation*} 
We rewrite  condition (\ref{LE}) as 
\begin{equation*}
\sum_{\gamma ^{\prime },\gamma _{1}^{\prime }}w(\gamma ,\gamma
_{1}|\gamma ^{\prime },\gamma _{1}^{\prime })\frac{f_{0}(\gamma
^{\prime })f_{0}(\gamma _{1}^{\prime })}{f_{0}(\gamma
)f_{0}(\gamma_{1})}=\sum_{\gamma ^{\prime },\gamma _{1}^{\prime
}}w(\gamma^{\prime},\gamma _{1}^{\prime }|\gamma,\gamma _{1}),
\end{equation*} 
and set for the sake of shortness
$\int\equiv\sum_{\gamma,\gamma_{1},\gamma ^{\prime},\gamma
_{1}^{\prime }}$. Then, for any function $f(\gamma)$, we have
\begin{equation*}
\int w(\gamma ,\gamma _{1}|\gamma ^{\prime },\gamma _{1}^{\prime
})\frac{f_{0}(\gamma ^{\prime })f_{0}(\gamma _{1}^{\prime
})}{f_{0}(\gamma)f_{0}(\gamma _{1})}f(\gamma )f(\gamma _{1})=\int
w(\gamma ^{\prime},\gamma_{1}^{\prime }|\gamma ,\gamma _{1})f(\gamma
)f(\gamma _{1}),
\end{equation*} 
or, after a change of variables,
\begin{equation*}
\int w(\gamma ^{\prime },\gamma _{1}^{\prime }|\gamma ,\gamma _{1})f(\gamma
)f(\gamma _{1})=\int w(\gamma ^{\prime },\gamma _{1}^{\prime }|\gamma
,\gamma _{1})\frac{f_{0}(\gamma )f_{0}(\gamma _{1})}{f_{0}(\gamma ^{\prime
})f_{0}(\gamma _{1}^{\prime })}f(\gamma ^{\prime })f(\gamma _{1}^{\prime }).
\end{equation*} 
Let  $\phi(\gamma)= \log \left[\frac{f_{0}(\gamma)}{f(\gamma)}\right]$. Then 
\begin{equation*}
\begin{split}
\frac{dH(f)}{dt} &=\sum_{\gamma}\frac{df(\gamma)}{dt}\phi(\gamma) \\
&=\int \phi(\gamma)\bigl[w(\gamma ,\gamma _{1}|\gamma ^{\prime
},\gamma_{1}^{\prime })f(\gamma ^{\prime })f(\gamma
_{1}^{\prime})-w(\gamma ^{\prime },\gamma _{1}^{\prime
}|\gamma,\gamma_{1})f(\gamma )f(\gamma_{1})\bigr] \\
&=\int\big[\phi(\gamma)-\phi(\gamma^{\prime})\bigr]
w(\gamma,\gamma_{1}|\gamma^{\prime},\gamma_{1}^{\prime})
f(\gamma^{\prime})f(\gamma_{1}^{\prime}) \\
&=\frac{1}{2}\int\bigl[\phi(\gamma)+\phi(\gamma_{1})-
\phi(\gamma^{\prime})-\phi(\gamma_{1}^{\prime})\bigr]w(\gamma ,\gamma
_{1}|\gamma^{\prime },\gamma _{1}^{\prime })f(\gamma ^{\prime
})f(\gamma _{1}^{\prime }). \\
\end{split}
\end{equation*}
Set for a while
\begin{equation*}
\begin{cases}
\displaystyle \xi =
\frac{f_{0}(\gamma)f_{0}(\gamma_{1})f(\gamma^{\prime})f(\gamma_{1}^{\prime})}
{f(\gamma)f(\gamma_{1})f_{0}(\gamma^{\prime})f_{0}(\gamma_{1}^{\prime})},
\\[0.5cm] \displaystyle \alpha =
\frac{f_{0}(\gamma^{\prime})f_{0}(\gamma_{1}^{\prime})}
{f_{0}(\gamma)f_{0}(\gamma_{1})},
\end{cases}
\end{equation*}
so that
\[\log \xi = 
\phi(\gamma)+\phi(\gamma_{1})-\phi(\gamma^{\prime})-\phi(\gamma_{1}^{\prime}).
\]
Then
\begin{equation*}
\frac{dH(f)}{dt}=\frac{1}{2}\int \alpha \xi \log\xi
w(\gamma,\gamma_{1}|\gamma^{\prime},\gamma_{1}^{\prime})f(\gamma)
f(\gamma_{1}).
\end{equation*}
On the other hand, from the LE condition, 
\begin{equation*}
\int \alpha \xi w(\gamma ,\gamma _{1}|\gamma ^{\prime },\gamma
_{1}^{\prime })f(\gamma )f(\gamma _{1})=\int \alpha w(\gamma ,\gamma
_{1}|\gamma^{\prime},\gamma_{1}^{\prime})f(\gamma )f(\gamma_{1}),
\end{equation*}
which yields
\begin{equation*}
\frac{dH(f)}{dt}=\frac{1}{2}\int (\xi \log \xi -\xi +1)\alpha
w(\gamma,\gamma_{1}|\gamma^{\prime },\gamma _{1}^{\prime
})f(\gamma)f(\gamma_{1})\geq 0,
\end{equation*} 
since $\xi \log \xi -\xi +1>0$ if $\xi >0$, due to the convexity of
$\xi\log\xi$.

\medskip\noindent
Assume now that for some $f_{0}>0$ the local equilibrium condition
holds. Then it holds also for any other stationary solution $f$, i.e.
satisfying $\frac{df}{dt}=0$. In fact, note that $\frac{dH(f)}{dt}>0$
if $f(\gamma )f(\gamma _{1})>0,w(\gamma ,\gamma _{1}|\gamma ^{\prime
},\gamma _{1}^{\prime })>0$ and $\xi \neq 1$. Also, if $f$ is a
stationary solution of equation (\ref {bolt_1}) then $\frac{df(\gamma
)}{dt}=0$ and hence $\frac{dH(f)}{dt}=0$. It follows that, for any
$\gamma ,\gamma _{1},\gamma ^{\prime },\gamma_{1}^{\prime }$ such that
$f(\gamma )f(\gamma _{1})>0$ and
$w(\gamma,\gamma_{1}|\gamma^{\prime},\gamma_{1}^{\prime })>0$, we have
$\xi=1$, that is
\begin{equation*}
\frac{f(\gamma ^{\prime })f(\gamma _{1}^{\prime })}{f(\gamma )f(\gamma _{1})} 
=\frac{f_{0}(\gamma ^{\prime })f_{0}(\gamma _{1}^{\prime })}{f_{0}(\gamma
)f_{0}(\gamma _{1})} \,.
\end{equation*} 
On the other hand, if $\frac{df(\gamma )}{dt}=0,f(\gamma )=0$ and
$w(\gamma,\gamma_{1}|\gamma ^{\prime },\gamma _{1}^{\prime })=0$, then
we get $f(\gamma^{\prime })f(\gamma _{1}^{\prime })=0$ as a
consequence of equation (\ref {bolt_1}). Thus, for any $\gamma ,\gamma
_{1}$, equation (\ref{LE}) holds. Any solution $f(t)$ of equation
(\ref{bolt_1}) as $t\rightarrow \infty $ tends to some stationary
solution $f_{\infty }$, which depends in general on the initial data
$f(0)$. In fact, from the proof of theorem \ref{TH1}, it follows that
$f$ is a stationary solution, i.e. $\frac{df}{dt}=0$, if and only if
$\frac{dH(f)}{dt}=0$ (provided that (\ref{bolt_1}) holds). This means
that $H(f)$ is a Lyapounov function. Consequently, the expected result
follows from the general theory of Lyapounov functions and the proof
of the theorem is terminated.
\end{proof}

%%%%%%%%%%%%%%%%%%%%%
\subsection{Fixed points and conservation laws}
Now we will prove that, for any two fixed points $f_{0},f$, the function $\log 
\frac{f}{f_{0}}$ is an additive conservation law. Consider the equation 
\begin{equation}
\frac{f(\gamma ^{\prime })f(\gamma _{1}^{\prime })}{f(\gamma )f(\gamma
_{1})} =
\frac{f_{0}(\gamma^{\prime})f_{0}(\gamma_{1}^{\prime})}{f_{0}(\gamma)f_{0}
(\gamma_{1})}
\label{cons_1}.
\end{equation} 
For $f_{0}=1$, we have 
\begin{equation}
\frac{f(\gamma^{\prime })f(\gamma_{1}^{\prime })}{f(\gamma )f(\gamma_{1})} 
= 1, \label{cons_2}
\end{equation} 
which shows that $\log f$ is an additive conservation law. Vice versa,
if there is a set $J$ of additive conservation laws such that
\begin{equation*}
\eta _{j}(\gamma )+\eta _{j}(\gamma _{1})=\eta _{j}(\gamma ^{\prime })+\eta
_{j}(\gamma _{1}^{\prime }),j\in J,
\end{equation*}
then, for any constants $c,c_{j}$, 
\begin{equation*}
f(\gamma )=c\prod_{j\in J}\exp (c_{j}\eta _{j}(\gamma ))
\end{equation*}
is a solution of (\ref{cons_2}). Note that additive conservation laws form a
linear space. Thus we have proved that any solution of (\ref{cons_2}) has
this form. In the general case (that is if $f_{0}\neq 1$), we have  
\begin{equation*}
\frac{f}{f_{0}}=c\prod_{j\in J}\exp (c_{j}\eta _{j}(\gamma )).
\end{equation*}
It is worth noticing that a nonzero additive conservation law for the
chain $\mathcal{L}_{M}$ is in fact unique, if the chains
$\mathcal{L}_{M}(U)$ are irreducible, for all $U$.

%%%%%%%%%%%%%%%%%%%%%%%%%%%%%%%
\section{Invariant measures for multitype models}
Here we will analyze some cases with $V>1$, when there exists an
invariant measure having a product form.

%%%%%%%%%%%%%%%
\subsection{Binary reactions without type change}
Let for any $v=1,\ldots,V$ a density $\rho _{v}(x)>0$ on $R_{+}$ be
given.  Assume only reactions $v,w\rightarrow v,w$\ are possible, so
that the $n_{v}$'s are conserved. Then one can introduce finite
particle Markov chains $\mathcal{L} _{n_{1},\ldots,n_{V}}$. Suppose in
addition that, for any couple of types $(v,w)$,
\begin{equation*}
\alpha _{vw}(T,T^{\prime })=\alpha _{vw}(T+T^{\prime }),
\end{equation*} 
which means that the rates depend only on the sum of energies.

\noindent
We need the following definition. Fix a pair $(v,w)$ of types and let
$\xi_{v},\xi _{w}$ be independent random variables with joint density
$\rho_{v}(x)\rho _{w}(y)$. Denote
\begin{equation*}
P_{\rho _{v}\rho _{w}}=P_{\rho _{v}\rho _{w}}(x,y|T)=P(\xi _{v}=x,\xi
_{w}=y|\xi _{v}+\xi _{w}=T)
\end{equation*} 
the corresponding conditional distributions, which will be called
\emph{canonical kernels} corresponding to the density array $(\rho
_{1},\ldots,\rho_{V})$.

\noindent
Let $\xi_{v,i},i=1,\ldots,n_{v},$ stand for the energy of the $i$-th
particle of type $v$.
\begin{theo}
  Fix an array $\rho _{1},\ldots,\rho _{V}$ and let a system of
  $\frac{V(V+1)}{2}$ reactions with canonical kernels $P_{\rho
    _{v}\rho _{w}}$ be given. Then, for any $n_{1},\ldots,n_{V}$, the
  invariant measures of $\mathcal{L}_{n_{1},\ldots,n_{V}}$ are such
  that the random variables $\xi _{v,i}$ have independent
  distributions equal to $\rho_{v}$. In the thermodynamic limit, for
  any initial concentrations of types $(c_{1},\ldots,c_{V})$ (here the
  concentrations of types do not change at all), the invariant energy
  distribution is unique and given by the independent densities $\rho
  _{v}$. Moreover, for any initial energy distribution, there is
  convergence to this invariant measure.

\noindent
Also, for any array $(\rho _{1},\ldots,\rho _{V})$ with arbitrary
rates $\alpha _{vw}(U)$, there is only one system of kernels for
which this array defines an invariant (product form) distribution,
these kernels being canonical kernels.
\end{theo}
\begin{proof} Any transition $v,v^{\prime }\rightarrow v,v^{\prime }$ 
  conserves $U$ and the related measures. Hence, as for the
  convergence, the argument is similar to that in the previous
  section.  The other statements follow directly from the definitions.
\end{proof}

\medskip\noindent
When $\rho _{v}(x)=\beta e^{-\beta x}$, the kernels are uniform on
$\left[ 0,T\right] $, as in the one type case study. Let $P^{\beta }$
denote such a kernel. An interesting situation depicted in the next
remark arises when
\begin{equation*}
\rho _{v}(x)=\beta _{v}\exp (-\beta _{v}x),
\end{equation*} 
with different $\beta _{v}$'s.
\begin{remark}
 All other cases can be reduced to the simplest one by the following  transformation. Given any density $\rho >0$ and any $\beta >0$, introduce the one to one mapping \thinspace $U=U(\rho ,\beta):R_{+}\rightarrow R_{+}$ such that, for any $x\in R_{+}$,
\begin{equation*}
\int_{0}^{x}\rho (y)dy=\int_{0}^{U(x)}\beta e^{-\beta y}dy.
\end{equation*} 
Then 
\begin{equation*}
P_{vw}= \bigl(U^{-1}(\rho _{v},\beta) ,U^{-1}(\rho_{w},\beta)\bigr)P^{\beta}
\bigl(U(\rho _{v},\beta) ,U( \rho_{w},\beta)\bigr).
\end{equation*}
\end{remark}

%%%%%%%%%%%%%
\subsection{Unary reactions}\label{ENERGY}
Now we want to tackle examples in which the $n_{v}$'s are not
conserved. Then, in general, only
\begin{equation*}
\mathcal{L}^{M}=\bigcup _{n_{1}+\cdots+n_{V}=M}\mathcal{L}_{n_{1},\ldots,n_{V}}
\end{equation*} 
is a Markov chain. In this subsection, we assume that unary reactions 
\begin{equation*}
v\rightarrow w
\end{equation*} 
can take place with rates $a_{vw}$. Such reactions could be interpreted
as isomer to isomer transformations. When $I_{v}\geq I_{w}$ the
reaction $ v\rightarrow w$ always occurs, and the kinetic energy $T$ of
the $v$-particle becomes the kinetic energy $I_{v}-I_{w}+T$ of the $w$-particle.
The reaction $ w\rightarrow v$ however occurs only if
$T-I_{v}+I_{w}\geq 0$, in which case the kinetic energy $T$ of
the $w$-particle becomes the kinetic energy $ T-I_{v}+I_{w}$ of the
$v$-particle.

\noindent 
Consider first the case without binary reactions. Define the following
\emph{one-particle} Markov chain: its states are all pairs $(v,T)$, that is
$M=1$. Moreover, assume that there are only two types. Let
$I_{1}<I_{2}$. Consider a pair of densities $\rho_1,\rho_2$, and
denote by $\xi_1,\xi_2$ the corresponding random variables. We call
this pair \emph{admissible} if the conditional density of $\xi_{1}-(I_{2}-I_{1})$, on the event $\{\xi _{1}>I_{2}-I_{1}\}$, is
equal to $\rho_{2}$. One example is $\rho_{1}=\rho_{2}=\beta\exp(-\beta x)$,
another being 
\begin{equation*}
\rho_{1}(x) = 
\begin{cases}
0, & \mathrm{for} \ x<I_{2}-I_{1}; \\[0.2cm] \rho _{2}(x-I_{2}+I_{1}),
& \mathrm{otherwise}.
\end{cases}
\end{equation*}
Any invariant measure on $\left\{ 1,2\right\} \times R_{+}$ can be
written as $\pi_{1}(1,\rho_{1})+\pi_{2}(2,\rho_{2})$ with positive
coeeficients $ \pi_{i}$ such that $\pi_{1}+\pi_{2}=1$. We have
 for $\pi_{1},\pi _{2}$ the following equations
\begin{equation*}
\pi_{1}Y_{1}a_{12}=\pi_{2}a_{21}, \quad
Y_{1}=\int_{I_{2}-I_{1}}^{\infty }\rho_{1}(x)dx .
\end{equation*}
This case exhibits the highest degree of reducibility, each class
containing one or two elements: there are plenty of invariant measures
-- but this is clearly a very unnatural situation. For an arbitrary
$M$ with only two types, we have the product of $M$ chains
$\mathcal{L}^{1}$, which again leads to a rather unnatural situation.

\noindent 
When there are $V>2$ types, each class also has a finite number of
elements. It is then possible to order the internal energies, assuming
for example
\begin{equation*}
I_{1}\leq I_{2}\leq\ldots\leq I_{V},
\end{equation*} 
and also $a_{vw}>0, \forall v,w$. If the full energy satisfies
$I_{m}<U<I_{m+1},m=1,\ldots,V$ (putting $I_{m+1}=\infty $) then there
are no possible jumps to the types $m+1,\ldots,V$, so that the process
evolves as a Markov chain $\mathcal{L}_{1,m}$ with state space
$1,\ldots,m$ and rates $a_{vw},v,w=1,\ldots,m$. Hence
$\mathcal{L}_{1,m}$ are restrictions of $\mathcal{L}_{1,V}$. For
$m=1$, it becomes a trivial one-point Markov chain. Let $\pi
_{m,v},v=1,\ldots,m$ denote the stationary probability of the state
$v$ in $\mathcal{L}_{1,m}$. We have $\pi _{1,1}=1$.

\noindent 
Note that, if at time $0$ the state is $(1,U)$ and $U$ has some
density $f(U)$ in $[I_{m},I_{m+1}]$, then the stationary distribution
is defined by $\pi_{m,v}$ and by the conditional density $f$ of the
full energy. Thus everything is defined by the rates $a_{vw}$ and by
$f(U)$, that is $\rho _{1}$. Moreover, these quantities can be chosen
arbitrarily. Setting for the sake of shortness
\begin{equation*}
\pi _{v}=\pi _{V,v},
\end{equation*}
we propose hereafter some examples.

\paragraph{\emph{Shifts}}
In this first example we take $\rho _{1}(x)=0$ if $x<I_{V}-I_{1}$.
Then each $\rho_{v}$ is just a shift of $\rho_{1}$.

\paragraph{\emph{Reversibility}}
Analogously, a system $(\rho _{1},\ldots,\rho _{V})$ of densities will
be said \emph{admissible} if the following condition holds: for any
$v$ the pair $(\rho _{v},\rho _{v+1})$ is admissible. Then it follows
that each pair of densities $(\rho_{i},\rho_{j}),i<j$, is admissible.

\begin{theo}
  If $I_{1}<\ldots<I_{V}$, all $\rho _{v}(x)$ are strictly positive
  and the system $(\rho _{1},\ldots,\rho _{V})$ is admissible then
  $\mathcal{L}^{M}$ is reversible.
\end{theo}
\begin{proof} Let $f_{v}(U)=\rho _{v}(U-I_{v})$ for $U\geq I_{v}$ 
  and $f_{v}(U)=0$ for $U<I_{v}$. We suppose the invariant
  distribution for the chain $\mathcal{L}^{M}$ has a product form,
  each factor being given by $\pi _{v}f_{v}(U)$. This means that for
  each $m$ and for $I_{m}\leq U <I_{m+1}$
\begin{equation*}
\sum_{i=1}^{m}\pi_{i}f_{i}(U)a_{ij}=\pi_{j}f_{j}(U)\sum_{i=1}^{m}a_{ji},
\end{equation*} 
for $j=1,\ldots,m$. Then admissibility means
\[
f_{i}(U) = 
\begin{cases}
A_{i}f_{1}(U), \ \mathrm{for} \ U\geq I_{i}, \\[0.2cm]
f_{i}(U) = 0, \ \mathrm{otherwise}. 
\end{cases}
\]
Hence 
\begin{equation*}
\sum_{i=1}^{m}\pi _{i}A_{i}a_{ij}=\pi _{j}A_{j}\sum_{i=1}^{m}a_{ji}, \quad 
1\leq j\leq m .
\end{equation*}
Putting $p_{i}=\pi _{i}A_{i}$, it follows that
\begin{equation*}
\sum_{i=1}^{m}p_{i}a_{ij}=p_{j}\sum_{i=1}^{m}a_{ji},\ 1\leq j\leq m,
\quad \forall m=1,\ldots,V.
\end{equation*}
The comparison of these equations for $m$ and $m+1$ yields
\begin{equation*}
p_{m+1}a_{m+1.j}=p_{j}a_{j,m+1}, \ \forall j=1,\ldots,m 
\end{equation*} 
and by induction we get
\begin{equation*}
p_{i}a_{ij}=p_{j}a_{ji}, \ \forall i,j ,
\end{equation*}
 which  implies the announced reversibility of $\mathcal{L}^{M}$.
 \end{proof}

\paragraph{\emph{Exponential}} In this third example, we also assume the 
system $(\rho_{1},\ldots,\rho_{V})$ of densities is admissible, and
moreover that, for some $\rho (T)$ and all $v$,
\begin{equation*}
\rho _{v}=\rho.
\end{equation*}
\begin{theo}
  Suppose $V\geq 3$, and that the quantities $I_{2}-I_{1}$ and
  $I_{3}-I_{2}$ are incommensurable. Then
\begin{equation*}
\rho (T)=\beta \exp (-\beta T),
\end{equation*} 
for some $\beta >0$.
\end{theo}
\begin{proof}
Admissibility implies that 
\begin{eqnarray*}
\rho_{2}(T) &=& A_{2}\rho _{1}(T+I_{2}-I_{1}),\\
\rho_{3}(T) &=& A_{3}\rho _{1}(T+I_{3}-I_{1}).
\end{eqnarray*} 
If $\rho _{1}=\rho _{2}=\rho _{3}=\rho$ then 
\begin{equation*}
\rho(T) = A_{2}\rho (T+x_{2}) = A_{3}\rho (T+x_{3}),
\end{equation*}
where $x_{i}=I_{i}-I_{1},\ i=2,3$ are incommensurable. But these last
two equations are compatible only if $A_{2}^{x_{3}}=A_{3}^{x_{2}}$ and
$\rho (T)=\beta \exp (-\beta T)$, with 
\begin{equation*}
\beta =\frac{\log A_{2}-\log A_{3}}{x_{2}-x_{3}}.
\end{equation*}
\end{proof}

\paragraph{\emph{Energy dependence}} 
In the fourth example, the rates $a_{vw}=a_{vw}(T)$ depend on the
energy of the input particle $v$. To construct a model which will be
needed later, consider a reversible Markov chain $\mathcal{V}_{1}$ on
$\left\{ 1,\ldots,V\right\} $ with stationary probabilities $p_{v}$
and rates $b_{vw}$. Thus
\begin{equation*}
p_{v}b_{vw}=p_{w}b_{wv}.
\end{equation*}
For reactions $v\rightarrow w$, define the reaction rates as
\begin{equation*}
a_{vw}(U)= 
\begin{cases}
0, & \mathrm{if} \ U<I_{w}, \\[0.2cm]
(U-I_{w})^{\alpha _{w}}b_{vw}, & \ \mathrm{otherwise}.
\end{cases}
\end{equation*}
 Note that these rates are close to zero if the kinetic energy
$T_{w}=U-I_{w}$ of the $w$-particle is close to zero. Letting 
$f_{v}(U)$ be the density of the full energy of the $v$-particle, the
reversibility condition writes
\begin{equation} \label{rev_5}
\pi_{v}f_{v}(U)a_{vw}(U)=\pi_{w}f_{w}(U)a_{wv}(U), 
\end{equation}
for $U > \max (I_{v},I_{w})$. We take as  density $f$ the
shifted $\Gamma $-distribution
\begin{equation} \label{gamma_1}
f_{v}(U) = 
\begin{cases}
  \displaystyle \frac{\beta ^{\nu _{v}}}{\Gamma
    (\nu_{v})}(U-I_{v})^{\nu_{v}-1}\exp
  [-\beta (U-I_{v})], \ \mathrm{if} \  U >I_{v} , \\[0.3cm]
  0, \mathrm{otherwise}.
\end{cases}
\end{equation}
Here $\nu_{v}=\alpha_{v}+1$. Then equation (\ref{rev_5}) becomes 
\begin{equation*}
\frac{\pi_{v}\beta^{\nu_{v}}}{\Gamma (\nu_{v})}e^{\beta I_{v}}b_{vw}= 
\frac{\pi_{w}\beta^{\nu_{w}}}{\Gamma (\nu_{w})}e^{\beta I_{w}}b_{wv} \,,
\end{equation*}
 showing  that the stationary probabilities $\pi _{v}$ of type $v$ are
equal to (up to a common factor)
\begin{equation} \label{diez}
p_{v}e^{-\beta I_{v}}\Gamma (\nu _{v})\beta ^{-\nu _{v}},
\end{equation} 
and the resulting Markov chain is reversible.

%%%%%%%%%%%%%%%%%%%%%%%%
\subsection{Binary reactions without energy dependence}
Let us suppose that $a_{vw}$ do not depend on energies, so that types
evolve independently of the energies. Thus at any time $t$, we will
have probabilities $p_{t}(n_{1},\ldots,n_{V})$. We will look for cases
when there exists an invariant measure on each $\mathcal{L}^{M}$,
defined by probabilities $p(n_{1},\ldots,n_{V})$, and independent
conditional distribution of energies
\begin{equation*}
\prod\limits_{v,i}\rho _{v,i},
\end{equation*} 
[given $n_{1},\ldots,n_{V}$], defined by densities
$\rho_{v.i}(x)=\rho_{v}(x)$.

\medskip\noindent 
Assume all $I_{v}$'s are equal, but any reaction $v,w\rightarrow
v^{\prime },w^{\prime }$ can occur and let now a reaction
$v,w\rightarrow v^{\prime },w^{\prime }$ be given. We again call
\begin{equation*}
P_{\rho _{v^{\prime }}\rho _{w^{\prime }}}(x,y|T)
\end{equation*} 
 the canonical kernel corresponding to the reaction $v,w\rightarrow
v^{\prime },w^{\prime }$ and we denote by $\rho _{vw}(T)$ the density
of $\xi _{v}+\xi _{w}$.

\begin{theo}
  Suppose an array $(\rho _{1},\ldots,\rho _{V})$ of
  densities is given, satisfying for any binary reaction
  $v,w\rightarrow v^{\prime },w^{\prime }$ the conditions
\begin{equation*}
\rho _{vw}(T)=\int_{x+y=T}\rho _{v}(x)\rho _{w}(y)dxdy=\int_{x+y=T}
\rho_{v^{\prime }}(x)\rho _{w^{\prime }}(y)dxdy = 
\rho_{v^{\prime }w^{\prime }}(T).
\end{equation*}
Assume also canonical kernels and that, as $t\rightarrow \infty$, the
limit of $p_{t}(n_{1},\ldots,n_{V})$ exists. Then there is an
invariant measure having these densities.
\end{theo}

%%%%%%%%%%%%%%%%%%%%%%%%%%%%%%%%%
\subsection{General binary reaction case}
Here the $I_{v}$'s can be different, but we assume only binary
reactions $v,w\rightarrow v^{\prime },w^{\prime }$ are possible.

%%%%%%%%%%%%%%
\subsubsection{Complete factorization}\label{BINFAC}
Denote $\hat{\imath}$ a pair of types $(v,w)$. Thus reaction
$v,w\rightarrow v^{\prime },w^{\prime }$ will be written as
$\hat{\imath}\rightarrow \hat{\jmath}$, where
$\hat{\imath}=(v,w),\hat{\jmath}=(v^{\prime },w^{\prime })$. We shall
use the analog of the third example with binary reactions. Consider a
Markov chain $\mathcal{V}_{1}\times \mathcal{V}_{1}$ on
$\left\{1,\ldots,V\right\} \times \left\{ 1,\ldots,V\right\}$ with
rates $b_{\hat{\imath} \hat{\jmath}}$, such that its stationary
distribution be a product form $p_{(v,w)}=p_{v}p_{w}$ and the chain be
reversible. We define, for \emph{vector particles}
$\hat{\imath}=(v,w)$, the energies $I_{\hat{\imath}}=I_{v}+I_{w}$ and
\begin{equation*}
f_{\hat{\imath}}(U)=(f_{v}\ast f_{w})(U),
\end{equation*}
where $f_{v},f_{w}$ are given in (\ref{gamma_1}). Thus
$f_{\hat{\imath}}(U)$  has also a shifted $\Gamma $-distribution with
parameters $I_{\hat{\imath}}=I_{v}+I_{w},\,\nu _{\hat{\imath}}=\nu
_{v}+\nu _{w},\,\beta $. The reversibility condition, with some
unspecified stationary probabilities $\pi _{\hat{\imath}}$, writes
\begin{equation}
\pi_{\hat{\imath}}f_{\hat{\imath}}(U)a_{\hat{\imath}\hat{\jmath}}(U) = 
\pi_{\hat{\jmath}}f_{\hat{\jmath}}(U)a_{\hat{\jmath}\hat{\imath}}(U),
\label{rev_6}
\end{equation}
where $U > \max (I_{\hat{\imath}},I_{\hat{\jmath}})$. Letting
\[a_{\hat{\imath}\hat{\jmath}}(U)=
\begin{cases} 
0, & \mathrm{if} \  U<I_{\hat{\jmath}} \,; \\[0.2cm]
(U-I_{\hat{\jmath}})^{\alpha _{\hat{\jmath}
}}b_{\hat{\imath}\hat{\jmath}} & \  \mathrm{otherwise}.
\end{cases}
\]
 Here
\begin{equation*}
\alpha _{\hat{\jmath}}=\nu _{\hat{\jmath}}-1=\nu _{v^{\prime }}+\nu
_{w^{\prime }}-1,
\end{equation*}
and the reversibility condition becomes
\begin{equation*}
\pi_{\hat{\imath}}\frac{\beta ^{\nu_{\hat{\imath}}}}{\Gamma (\nu
_{\hat{\imath}})}e^{\beta
I_{\hat{\imath}}}b_{\hat{\imath}\hat{\jmath}}=\pi_{\hat{\jmath}}
\frac{\beta ^{\nu _{\hat{\jmath}}}}{\Gamma (\nu_{\hat{\jmath}})} e^{\beta
I_{\hat{\jmath}}}b_{\hat{\jmath}\hat{\imath}}.
\end{equation*}
We are looking for solutions $\pi _{(v,w)}=\pi _{v}\pi _{w}$, since we
are primarily interested in factorizable invariant distributions. To
this end, we assume in addition that, for any binary reaction
$v,w\rightarrow v^{\prime },w^{\prime }$, the condition
\begin{equation*}
\nu _{v}+\nu _{w}=\nu _{v^{\prime }}+\nu _{w^{\prime }}
\end{equation*}
is fulfilled. Then 
\begin{equation*}
\pi _{v}\pi _{w}\beta ^{\nu _{v}+\nu _{w}}e^{\beta I_{v}}e^{\beta
I_{w}}b_{\hat{\imath}\hat{\jmath}}=\pi _{v^{\prime }}\pi _{w^{\prime
}}\beta ^{\nu _{v^{\prime }}+\nu _{w^{\prime }}}e^{\beta I_{v^{\prime
}}}e^{\beta I_{w^{\prime }}}b_{\hat{\jmath}\hat{\imath}},
\end{equation*}
and up to a common factor, the solution of this system has the form
\begin{equation}
\pi _{v}=p_{v}e^{-\beta I_{v}}\beta ^{-\nu _{v}}  \label{okoshko}
\end{equation}

%%%%%%%%%%%%%%%
\subsubsection{Unary reactions included}
Let $V=\bigcup _{\alpha }V_{\alpha }$ be a disjoint union of sets
$V_{\alpha }$ of isomers. Thus, we assume that unary reactions
$v\rightarrow w$ are allowed only if $v$ and $w$ belong to the same
$V_{\alpha }$. The energy dependence of unary reactions will be
defined in the same way as in section \ref{ENERGY}, but additionally
we take $\nu _{v}$ to be constant on each $V_{\alpha }$, in other
words $\nu _{v}=\nu _{w}$ for any two isomers $v,w\in V_{\alpha }$. 

\medskip\noindent
We consider the same binary reactions as in section \ref{BINFAC}, with
the assumption that they are concordant with unary reactions in the
following sense: $ p_{(v,w)}=p_{v}p_{w}$ are such that, for any
$\alpha $, the probabilities $p_{v}$ have the form given in section
\ref{ENERGY} up to a constant factor.

\begin{theo}
  If the previous conditions are fulfilled, then  formula
  (\ref{okoshko}) gives the factorized reversible invariant
  distribution, both for binary and unary reactions.
\end{theo}
\begin{proof} It suffices to compare the formulae (\ref{okoshko}) and
(\ref{diez}), remarking that the factor $\Gamma (\nu _{v})$ in
(\ref{diez}) can be omitted, since $\nu _{v}$ is constant on
$V_{\alpha }$.
\end{proof}

%%%%%%%%%%%%%%%%%%%%%%%%%%
%% Name and Affiliation %%
%%%%%%%%%%%%%%%%%%%%%%%%%%

\add{Guy Fayolle}
{INRIA Rocquencourt -- Domaine de Voluceau BP 105\\
78153 Le Chesnay, France. \quad Guy.Fayolle@inria.fr}
\medskip
\add{Vadim Malyshev}
{INRIA Rocquencourt -- Domaine de Voluceau BP 105\\
78153 Le Chesnay, France. \quad Vadim.Malyshev@inria.fr}
\medskip
\add{Serguei Pirogov} {IPPI -- Russian Academy of Sciences 
\\19 Bolshoi Karetny -- 101447 Moscow, Russia.\\ Work partially
supported by RFBR grant 02-01-01276.}

\end{document}